\documentclass[reqno]{amsart}
\usepackage{amssymb, stmaryrd}
\usepackage{mathrsfs}
\usepackage{amscd}
\usepackage{pdfsync} 

\newtheorem{theorem}{Theorem}[section]
\newtheorem{lemma}[theorem]{Lemma}

\newtheorem{example}[theorem]{Example}

\newtheorem{proposition}[theorem]{Proposition}
\newtheorem{corollary}[theorem]{Corollary}

\newcommand{\FF}{\overline{\boldsymbol{F}}}

\usepackage{color}  
\definecolor{darkgreen}{rgb}{0.03, 0.5, 0.03}

 \newcommand{\co} {\boldsymbol c}
 \newcommand{\Na} {\mbox{\rm Na}}
  \newcommand{\Kr} {\mbox{\rm Kr}}
   \newcommand{\Zar} {\mbox{\rm Zar}}
  \newcommand{\Max} {\mbox{\rm Max}}
    \newcommand{\Spec} {\mbox{\rm Spec}}
      \newcommand{\QMax} {\mbox{\rm QMax}}
        \newcommand{\QSpec} {\mbox{\rm QSpec}}

                  \newcommand{\Clinv}{\mbox{\texttt{Cl}}^{\mbox{\tiny\emph{\texttt{inv}}}}}
                  \newcommand{\Clinvrm}{\mbox{\emph{\texttt{Cl}}}^{\mbox{\tiny\emph{\texttt{inv}}}}}

\newcommand{\Fb}{\boldsymbol{\overline{F}}}
\newcommand{\f}{\boldsymbol{{f}}}

  \newcommand{\stf} {\star{_{\!{_f}}}}
    
    \newcommand{\stt} {\widetilde{\star}}

\begin{document}

\title
[SHARPNESS AND SEMISTAR  OPERATIONS]{SHARPNESS AND SEMISTAR  OPERATIONS\\ IN PR\"UFER-LIKE DOMAINS}

\author{Marco Fontana}
\address{Dipartimento di Matematica e Fisica, Universit\`a degli Studi Roma Tre, Largo San Leonardo Murialdo, 1, 00146 Roma, Italy}
\email{fontana@mat.uniroma3.it}

\author{Evan Houston}
\address{Department of Mathematics and Statistics, University of North Carolina at Charlotte,
Charlotte, NC 28223 U.S.A.}
\email{eghousto@uncc.edu}

\author{Mi Hee Park}
\address{Department of Mathematics,
Chung-Ang University, Seoul  06974, Korea}
\email{mhpark@cau.ac.kr}

\thanks{ The first-named author was partially supported   by {\sl GNSAGA} of {\sl Istituto Nazionale di Alta Matematica}.}

\thanks{{The second-named author was supported by a grant from the Simons Foundation (\#354565).}}

\thanks{The third-named author was supported by the National Research Foundation of Korea (NRF) grant funded by the Korea government(MSIP) (No. 2015R1C1A2A01055124).}

\date{\today}
%

%
%
 \begin{abstract}   %
Let $\star$ be a semistar operation on a domain $D$, $\stf$ the finite-type semistar operation associated to $\star$, and $D$ a Pr\"ufer $\star$-multiplication domain (P$\star$MD).  For the special case of a Pr\"ufer domain (where $\star$ is equal to the identity semistar operation), we show that a nonzero prime $P$ of $D$ is sharp, that is, that $D_P \nsupseteq \bigcap  D_M$, where the intersection is taken over the maximal ideals $M$ of $D$ that do not contain $P$, if and only if two closely related spectral semistar operations on $D$ differ.  We then give an appropriate definition of $\stf$-sharpness for an arbitrary P$\star$MD $D$ and show that a nonzero prime $P$ of $D$ is $\stf$-sharp if and only if its extension to the $\star$-Nagata ring of $D$ is sharp.  Calling a P$\star$MD $\stf$-sharp ($\stf$-doublesharp) if each maximal (prime) $\stf$-ideal of $D$ is sharp, we also prove that such a $D$ is $\stf$-doublesharp if and only if each $(\star, t)$-linked overring of $D$ is $\stf$-sharp.
\end{abstract}

\maketitle


 \section*{Introduction}

A nonzero prime ideal $P$ of a Pr\"ufer domain $D$ is said to be \emph{sharp} if $\bigcap \{D_M \mid M \in \Max(D),\  P \nsubseteq M\} \nsubseteq  D_P$.  In \cite{G-66} Gilmer showed that an almost Dedekind domain with all maximal ideals sharp must be a Dedekind domain.  Then Gilmer and Heinzer \cite{GH-67} made a more thorough study of sharpness in Pr\"ufer domains, proving \cite[Theorem 3]{GH-67} that in a Pr\"ufer domain $D$ all nonzero primes are sharp if and only if in each overring of $D$ all maximal ideals are sharp.  The primary goal of this work is to extend the study of sharpness to Pr\"ufer $\star$-multiplication domains (P$\star$MDs) $D$, where $\star$ is an arbitrary semistar operation  (definition recalled below) on $D$.

At this point it is helpful to recall a result of Griffin: An integral domain $D$ is a P$v$MD (where $v$ is the ordinary star operation on $D$ and $t$ is the canonically associated finite-type star operation associated to $v$) if and only if $D_M$ is a valua\-tion domain for each maximal prime $t$-ideal $M$ of $D$. In   Section 1, we provide background on semistar operations and P$\star$MDs and establish a few new results.  In particular, we show (Lemma~\ref{lemma3}) that if $D$ is a P$\star$MD (for some semistar ope\-ration $\star$ on $D$), then $R:=D^{\tilde{\star}}$ is an ``ordinary'' P$v$MD and $D^{\tilde{\star}}=\bigcap \{D_{Q \cap R} \mid Q \text{ is a maximal $t$-ideal of $R$}\}$ (where the semistar operation $\tilde{\star}$ is a spectral semistar operation naturally related to $\star$ and is defined below).

In Section 2 we give a semistar characterization of sharpness in Pr\"ufer domains, and we show that this can be extended to the P$\star$MD setting. It is well-known that $D$ is a P$\star$MD if and only if the associated Nagata ring, defined by $\Na(D,\star):=  \{f/g \mid f, g \in D[X],\ g \ne 0,\ {\boldsymbol{c}}(g)^{\star}=D^{\star}\}$ (where $X$ is an indeterminate and ${\boldsymbol{c}}(g)$ denotes the \emph{content} of $g$, that is, the ideal generated by the coefficients of $g$) is a Pr\"ufer domain \cite[Theorem 3.1]{FJS-03}.  This provides an important tool for achieving our main results in Section 3.  In our main theorem, we give a bijection between a certain set of overrings of a P$\star$MD and the entire set of overrings of the associated $\star$-Nagata ring of $D$, and, in Corollary~\ref{c:##}, we obtain an extension of the result of Gilmer and Heinzer mentioned above to P$\star$MDs.

\bigskip
\section{Background results}

Let $D$ be an integral domain with quotient field $K$. Let
$\boldsymbol{\overline{F}}(D)$ denote the set of all nonzero
$D$--submodules of   $K$, and let $\boldsymbol{F}(D)$ be the set of
all nonzero fractional ideals of $D$, i.e., $E \in
\boldsymbol{F}(D)$ if $E \in \boldsymbol{\overline{F}}(D)$ and
there exists a nonzero $d \in D$ with $dE \subseteq D$. Let
$\f(D)$ be the set of all nonzero finitely generated
$D$--submodules of $K$. Then, obviously $\f(D)
\subseteq \boldsymbol{F}(D) \subseteq
\boldsymbol{\overline{F}}(D)$.

Following Okabe-Matsuda \cite{o-m}, a \emph{semistar operation} on $D$ is a map $\star:
\boldsymbol{\overline{F}}(D) \to \boldsymbol{\overline{F}}(D),\ E \mapsto E^\star$,  such that, for all $x \in K$, $x \neq 0$, and
for all $E,F \in \boldsymbol{\overline{F}}(D)$, the following
properties hold:
\begin{enumerate}
\item[$(\star_1)$] $(xE)^\star=xE^\star$;
 \item[$(\star_2)$] $E
\subseteq F$ implies $E^\star \subseteq F^\star$;
\item[$(\star_3)$] $E \subseteq E^\star$ and $E^{\star \star} :=
\left(E^\star \right)^\star=E^\star$.
\end{enumerate}

A \emph{(semi)star operation} is a semistar operation that,
restricted to $\boldsymbol{F}(D)$,  is a star operation (in the
sense of \cite[Section 32]{G}). It is easy to see that a
semistar operation $\star$ on $D$ is a (semi)star operation if and
only if $D^\star = D$.

 If $\star$ is a semistar operation on $D$, then we can
consider a map\ $\star_{\!_f}: \boldsymbol{\overline{F}}(D) \to
\boldsymbol{\overline{F}}(D)$ defined, for each $E \in
\boldsymbol{\overline{F}}(D)$, as follows:

\centerline{$E^{\star_{\!_f}}:=\bigcup \{F^\star\mid \ F \in
\boldsymbol{f}(D) \mbox{ and } F \subseteq E\}$.}

\noindent It is easy to see that $\star_{\!_f}$ is a semistar
operation on $D$, called \emph{the  finite
type semistar operation associated to $\star$} or \emph{the semistar operation of finite
type associated to $\star$}.  Note that, for each $F \in
\boldsymbol{f}(D)$, $F^\star=F^{\star_{\!_f}}$.  A semistar
operation $\star$ is called a \emph{semistar operation of finite
type} (or,  \emph{finite
type semistar operation})  if $\star=\star_{\!_f}$.  It is easy to see that
$(\star_{\!_f}\!)_{\!_f}=\star_{\!_f}$ (that is, $\star_{\!_f}$ is
of finite type).

If $\star_1$ and $\star_2$ are two semistar operations on $D$, we
say that $\star_1 \leq \star_2$ if $E^{\star_1} \subseteq
E^{\star_2}$, for each $E \in \Fb(D)$, equivalently, if
$\left(E^{\star_{1}}\right)^{\star_{2}} = E^{\star_2}=
\left(E^{\star_{2}}\right)^{\star_{1}}$, for each $E \in
\Fb(D)$.
Obviously, for each semistar operation  $\star$,
we
have $\star_{\!_f} \leq \star$. Let $d_D$ (or, simply, $d$)  be the {\it identity (semi)star operation on $D$.}
  Clearly, $d \leq \star$ for all semistar  operations $\star$ on $D$.

We say that a nonzero ideal $I$ of $D$ is a
\emph{quasi-$\star$-ideal} if $I^\star \cap D = I$, a
\emph{quasi-$\star$-prime  ideal} if it is a prime quasi-$\star$-ideal,
and a \emph{quasi-$\star$-maximal  ideal} if it is maximal in the set of
all   proper   quasi-$\star$-ideals. A quasi-$\star$-maximal ideal is  a
prime ideal. It is possible  to prove that each  proper   quasi-$\star_{_{\!
f}}$-ideal is contained in a quasi-$\star_{_{\! f}}$-maximal
ideal.  More details can be found in \cite[page 4781]{FL3}. We
will denote by $\QMax^{\star}(D)$  (respectively, $\QSpec^\star(D)$) the set of all
quasi-$\star$-maximal ideals  (respectively, quasi-$\star$-prime ideals) of $D$.
When $\star$ is a (semi)star operation, the notion of  quasi-$\star$-ideal coincides with the ``classical'' notion of  \it  $\star$-ideal \rm (i.e., a nonzero ideal $I$ such that $I^\star = I$).

\bigskip

We say that
$\,\star\,$ is
{\it a stable semistar operation on $\,D\,$} if
$$
(E \cap F)^\star = E^\star \cap F^\star,  \;\, \textrm {  for all} \;
E,F \in\FF(D) \,.
$$

For each $Q \in \Spec(D)$, let  $\texttt{s}_Q$ be the semistar operation (of finite type) on $D$ defined as follows, for each $E \in
\Fb(D)$:
$$
E^{\texttt{s}_Q} := ED_Q\,.
$$
Let $Y$ be a subset of $\Spec(D)$ and let
 $\texttt{s}_Y$ be the semistar operation  on $D$ defined as follows, for each $E \in
\Fb(D)$:
$$
E^{\texttt{s}_Y} := \bigcap \{ED_Q\mid Q \in Y\}\,.
$$
If $Y = \emptyset$, then $\texttt{s}_{\emptyset}$ is the trivial semistar operation defined by $
E^{\texttt{s}_{\emptyset}} := K$, for each $E \in
\Fb(D)$.

A semistar operation of the type $\texttt{s}_Y$, for some $Y \subseteq \Spec(D)$, is called a {\it spectral semistar operation} on $D$.
As a consequence of flatness, it is easy to see that any spectral semistar operation is stable.

It is obvious that  $Y' \subseteq Y'' (\subseteq \Spec(D))$ implies  ${\texttt{s}_{Y''}} \leq {\texttt{s}_{Y'}} $.

 Set $
\stt  :=
\texttt{s}_{\tiny{\QMax}^{\stf}(D)}$. The star operation
$\stt$ is called {\it the spectral semistar
operation of finite type associated to $\,\star\,$.}
It is known that  $\stt  \leq
\stf$ and $\QMax^{\stf}(D) = \QMax^{\stt}(D)$.
  Note that the finite type semistar operation associated to a stable semistar operation is not necessarily stable. On the other hand, it is well known that finite type stable semistar operations coincide with finite type spectral semistar operations (see, for instance, \cite[page 2952]{FiFoSp}).

 In the following, we collect some of the properties concerning the relation between $Y\subseteq \Spec(D)$ and $\texttt{s}_Y$. Let $Y^{\mbox{\tiny\texttt{gen}}} := \{ P\in \Spec(D) \mid  P\subseteq Q, \mbox{ for some } Q \in Y\}$ and let $\Clinv(Y)$ be the closure of $Y$ in the inverse topology of $\Spec(D)$, i.e., $\Clinv(Y):=\ \bigcap \{ D(J) \mid D(J) \supseteq Y, \,  J   \in \f(D)\}$.

\begin{lemma}\label{spectral}
Let $D$ be an integral domain and let $Y, Y', Y''\subseteq \Spec(D)$.
\begin{itemize}
\item[\rm(1)]  $\emph{\texttt{s}}_Y$ is of finite type if and only if $Y$ is quasi-compact.

\item[\rm(2)] $\emph{\texttt{s}}_{Y'} = \emph{\texttt{s}}_{Y''}  \Leftrightarrow  Y'^{\mbox{\tiny\emph{\texttt{gen}}}} = Y{''}^{\mbox{\tiny\emph{\texttt{gen}}}}$.

\item[\rm(3)] $\widetilde{ \emph{\texttt{s}}_{Y'} } = \widetilde{ \emph{\texttt{s}}_{Y''}} \Leftrightarrow  {\Clinvrm}(Y') = {\Clinvrm}(Y'')$.

\end{itemize}
%
%
%
%
%
\end{lemma}

For the proof, see \cite[Corollary 4.4 and Proposition 5.1]{FiSp} and \cite[Proposition 4.1]{FiFoSp}.

\medskip

A generalization of the classical Nagata ring construction was
considered by  {Kang} (1987 \cite{K-87} and 1989
\cite{K-89}).
This construction   has been generalized to the semistar setting:
Given {{\sl any} integral
domain} $\,D\,$ and
{{\sl any}
semistar operation $\,\star\,$} on
$\,D\,,$ \ we define {\it the semistar Nagata ring \rm
} as
follows:
$$
{\textstyle\rm Na}(D,\star) :=\left\{f/g \,  \mid \, f, g \in
D[X]
\,, \; g\neq 0\,,\; {\boldsymbol
c}(g)^\star =D^\star \right\}\,.
$$

Note that $\,{\textstyle\rm Na}(D,\star) =
{\textstyle\rm
Na}(D,\stf)$\,.\
Therefore, the assumption $\,\star = \stf\,$ is not really
restrictive when
considering semistar Nagata rings.
\medskip

If $\,\star = d\,$ is the identity
(semi)star operation  on $D$, then:
$$
{\textstyle\rm Na}(D,d) = D(X)\,.
$$

	Some results on \sl star \rm Nagata rings proved by
{Kang} in 1989 are
generalized to the semistar setting in the following:

\begin{lemma}\label{lemma-Na}
 Let $\,\star\,$ be a nontrivial semi\-star
operation   on an integral domain $\,D\,$. \ Set:\\
\centerline{
 $N(\star) := N_D(\star) := \{h \in D[X] \, |\;
\boldsymbol{c}(h)^\star =
 D^\star \}$\,.}
$\mbox{  }$

\begin{itemize}
\item[\rm(1)] $N(\star) = D[X] \setminus \bigcup \{Q[X] \; | \;\, Q
\in
\QMax^{\stf}(D) \}\,$ is a saturated multiplicatively closed
subset of
$\,D[X]\,$ and $\, N(\star) = N(\stf)\,$.

\item[\rm(2)] \Max$(D[X]_{N(\star)}) = \{Q[X]_{N(\star)} \; |\;\,
Q
\in
\QMax^{\stf}(D)\}\,.$

\item[\rm(3)] ${\textstyle\rm Na}(D,\star) = D[X]_{N(\star)} =
\bigcap\{ D_Q(X) \,| \, Q \in
 \QMax^{\stf}(D) \}\,.$

\item[\rm(4)]  $\QMax^{\stf}(D)$ coincides with the canonical
image  in
$\Spec(D)$   of the maximal  spectrum of
$\,{\textstyle\rm Na}(D,\star)\,;$ \,
  \sl i.e.,
$\QMax^{\stf}(D) = \{ M \cap D \, | \, M \in $
\Max$({\textstyle\rm Na}(D,\star)) \}\,. $

\item[\rm(5)]   For each $E \in \FF(D)$, $E^{\tilde{\star}} = E{\textstyle\rm Na}(D,\star)
\cap K$\,.

\end{itemize}
\end{lemma}

For the proof see \cite[Theorem 3.1 and Proposition 3.4]{FL3}.

\medskip

\rm  If {$\,\star\,$ is {\sl
any} semistar
operation}  on {\sl any}  integral domain
$\,D\,$,\
then we define {\it the Kronecker function ring
of $\,D\,$
with respect to the semistar operation $\,\star\,$} \rm by:
$$
\begin {array} {rl}
\mbox{Kr}(D,\star) := \{ f/g  \, \ |& \, f,g \in D[X], \ g
\neq 0,
\;
\mbox{ and there exists }\\ & \hskip -4pt h \in D[X] \setminus \{0\}
\; \mbox{
with } (\boldsymbol{c}(f)\boldsymbol{c}(h))^\star
\subseteq (\boldsymbol{c}(g)\boldsymbol{c}(h))^\star \,\}.
\end{array}
$$

 This definition, given in  \cite[Theorem 5.1]{FL1},  leads to a natural extension
of the ``classical'' Kronecker function ring.
In order to relate this general construction with the  Kronecker function ring as defined by Krull  (see, for instance, \cite[page 401]{G}), we recall that it is possible to associate to an
arbitrary semistar ope\-ration $\star$ an \texttt{eab} semistar operation of finite type as follows, for each $ F \in \boldsymbol{f}(D)$ and
for each  $E \in {\overline{\boldsymbol{F}}}(D)\,:
$
$$
\begin {array} {ll}
F^{\star_a} &:= \cup\{((FH)^\star:H^\star) \; \ | \; \, \; H \in
\boldsymbol{f}(D)\}\,, \\
E^{\star_a} &:= \cup\{F^{\star_a} \; | \; \, F \subseteq E\,,\; F \in
\boldsymbol{f}(D)\}\,,
\end{array}
$$
 \rm(for the definition of \texttt{eab} operation see  \cite[page 394]{G} and \cite[Definition 2.3]{HK1}).
 The previous
construction is essentially due to P. Jaffard
(1960)    \cite[Chapitre II, \S 2]{J}
and {F. Halter-Koch} (1997, 1998)  \cite[Section 6]{HK1},
\cite[Chapter 19]{HK2}.

\smallskip

Obviously,
 $(\star_{_{\!f}})_{a}= \star_{a}$  and so,
when $\star =
\star_{_{\!f}}$, then $\star$ is \texttt{eab} if and only if $\star =
\star_{a} $  \cite[Proposition 4.5]{FL1}.

\medskip

Let  $D$ be an integral domain  and set $\Zar(D):= \{ V\mid V \mbox{ is a valuation overring of }$ $ D \}$. It is well known that $\Zar(D)$ can be equipped with   the {\it Zariski topology}, i.e., the topology having, as subbasic open subspaces, the subsets $\Zar(D[x])$ for $x$ varying in $K$.  The set  $\Zar(D)$, endowed with the Zariski topology, is often called the  {\it Riemann-Zariski space of }$D$  \cite[Chapter VI, \S 17, page 110]{ZS}.
Since it is a spectral space, $\Zar(D)$ can be also endowed with the inverse topology, that is the topology having as a subbasis for the closed sets the quasi-compact open subspaces in the Zariski topology.

Given a family $\mathscr{V}$ of valuation overrings of an integral domain $D$, we can define a semistar operation $\wedge_{\mathscr{V}}$ on $D$, by setting
$
E^{\wedge_{\mathscr{V}} }:= \bigcap \{EV\mid V \in \mathscr{V} \}$ for each $E \in \Fb(D)$.
Clearly,  $\wedge_{\mathscr{V}}$  is an \texttt{eab} semistar operation and it is known that ${\wedge_{\mathscr{V}} }$ is a semistar operation of finite type if and only if $\mathscr{V}$ is a quasi-compact subspace of $\Zar(D)$ \cite[Proposition 4.5]{FiSp}.

\medskip

 We also need to recall the notion of $\star$-valuation overring, considered by P. Jaffard
(1960)   \cite[page 46]{J} (see also  Halter-Koch  (1997)
\cite[Chapters 15 and 18]{HK2}).

For a domain
$\,D\,$ and a semistar operation $\,\star\,$ on $\,D\,$, \ we say
that
a
valuation overring $\,V\,$ of $\,D\,$ is {\it a
$\,\star$--valuation overring of $\,D\,$} \rm provided $\,F^\star
\subseteq FV\,,$ \ for each $\,F \in \boldsymbol{f}(D)\,$.   Note
that,
by definition the $\,\star$--valuation overrings coincide with the
$\stf$--valuation overrings.

We collect in the following lemma some properties needed later.

\begin{lemma}\label{Kr} \sl Let $\,\star\,$ be a semi\-star
 operation on
an integral domain $\,D\,$ with quotient field $\,K\,.$ Then:
\begin{itemize}
\item[{\rm (1)}] $\Na(D,\star) \subseteq
\Kr(D,\star)\,.$

\item[{\rm (2)}]  $V$ is a $\star$--valuation overring of $D$ if and
only if $V(X)$ is a valuation overring of $\Kr(D, \star)$.
\newline The map \
$W \mapsto W \cap K$ \ establishes a bijection between the set of all
valuation overrings of $\Kr(D, \star)$ and the set of all the
$\star$--valuation overrings of $D$.

\item[{\rm (3)}]  ${\textstyle\rm Kr}(D,\star) = {\textstyle\rm
Kr}(D,\stf) = {\textstyle\rm Kr}(D,\star_a) = \cap \{V(X) \mid
V$
is a $\star$--valuation
overring of $ D\}\,$ is a B\'ezout domain
with
quotient field $\,K(X)\,$.

\item[{\rm (4)}]  $E^{\star_a} = E{\Kr}(D,\star) \cap
K = \cap \{EV \mid  V   \mbox{is a $\star$--valuation
overring of }$ $ D\}\,,$ for each $\,E \in
\boldsymbol{\overline{F}}(D)\,.$
\end{itemize}
\end{lemma}
\vskip -4pt For the proof see  \cite[Theorem 3.11]{FL1}, \cite[Theorem
3.5]{FL2}, \cite[Proposition 4.1]{FL3}.

\bigskip

 Let $\star$ be a
semistar operation on an integral domain $D$.
Note that, from the previous Lemma \ref{Kr}(4), it follows that $\star_a= \wedge_{\mathscr{V}^\star(D)}$, where $\mathscr{V}^\star(D)$ denotes  the set of all $\star$--valuation
overrings of  $ D$.

\smallskip

 Recall that a \it Pr\"ufer
$\star$--multiplication domain
 \rm (for short, a \it  P$\star$MD\rm ) \ is an integral domain
$D$ such that $ (FF^{-1})^{\star_{_{\!f}}} =
D^{\star_{_{\!f}}} \ (= D^\star)$\;  (i.e.,\  $F$ is \it
 ${\star_{_{\!f}}}$--invertible\rm ) for each $F \in \boldsymbol{f}(D)$. Clearly, the notions of P$\star$MD and P$\stf$MD coincide and, given   two semistar operations $\star_1 \leq \star_2$ on $D$, if $D$ is a P$\star_1$MD then $D$ is also  a P$\star_2$MD.

In the following lemma, we collect some properties of Pr\"ufer
$\star$--multiplication domains.

\begin{lemma} \label{Na=Kr}
Let
$D$ be an integral
domain   with quotient field $K$ and $\star$ a semistar operation on $D$.
\begin{enumerate}
\item[(1)]
 The following are
equivalent:
\begin{enumerate}
\item[{\rm (i)}] $D$ is a P$\star$MD.

\item[{\rm (ii)}]
$\Na(D,\star)$ is a Pr\"{u}fer domain.

\item[{\rm (iii)}]
$\Na(D,\star)=\Kr(D,{\star})$\,.

\item[{\rm (iv)}]
$\tilde{\star} =\star_{a}$\,.

\item[{\rm (v)}]
${\star_{_{\!f}}} $ is stable and \texttt{eab}.
\item [{\rm (vi)}] Each ideal of $\Na(D,\star)$ is extended from an ideal of $D$.
\end{enumerate}
\noindent In particular, if $D$ is a P$\star$MD, then ${\star_{_{\!f}}}=\tilde{\star}$ and so  $D$ is a P$\star$MD if and only if it is a
P$\tilde{\star}$MD.

\item[{(2)}] Assume that $D$ is a P$\star$MD.  The contraction map $\Spec(\Na(D,\star)) \to \linebreak \QSpec^{\stf}(D)$,
 $Q \mapsto Q \cap D$, is a bijection with inverse map $P \mapsto P\Na(D,\star)$.

\end{enumerate}
\end{lemma}

\begin{proof}
\vskip -4pt For the proof of  (1) (i)-(v),  cf.   \cite[Theorem 3.1 and Remark
3.1]{FJS-03}.

 Assume that (iii) holds, let $J$ be an ideal of $\Na(D,\star)$, and write $J=I\Na(D,\star)$ for some ideal $I$ of $D[X]$.  We claim that $J= \boldsymbol{c}(I)\Na(D,\star)$.  To verify this, let $f \in I$.  By \cite[Lemma 2.5(e)]{FJS-03},
 $\boldsymbol{c}(f)\Na(D,\star)=\boldsymbol{c}(f)\Kr(D,\star)=f\Kr(D,\star)=f\Na(D,\star) \subseteq J$.  It follows that $\boldsymbol{c}(I)\Na(D,\star) \subseteq J$.  The reverse inclusion is clear.  Thus (iii) implies (vi).

 Now assume (vi), and let $A$ be a  nonzero finitely generated ideal of $\Na(D,\star)$.  Then by assumption we may write $A=I\Na(D,\star)$ for some ideal $I$ of $D$, and we may assume that $I$ is finitely generated.  Choose $f \in D[X]$ with $\boldsymbol{c}(f)=I$. Again by assumption, we may write $f\Na(D,\star)=J\Na(D,\star)$ for some ideal $J$ of $D$.  Now let $M$ be a maximal ideal of $\Na(D,\star)$; we have $M=Q\Na(D,\star)$ for some $Q \in \QMax^{\stf}(D)$ (Lemma~\ref{lemma-Na}).  Localizing at $M$, we obtain $fD_Q(X)=JD_Q(X)$.  By \cite[Theorem 1]{a}, we must then have $AD_Q(X)=\boldsymbol{c}(f)D_Q(X)=fD_Q(X)$.  Hence $A$ is locally principal and therefore invertible.  It follows that $\Na(D,\star)$ is a Pr\"ufer domain, and we have (vi) implies (ii).

(2)  Let $Q \in {\Spec}(\Na(D,\star))$. Then $Q=(Q\cap D)\Na(D, \star)$ and $(Q\cap D)^{\stf}\cap D=(Q\cap D)^{\tilde\star}\cap D=(Q\cap D)\Na(D, \star)\cap K\cap D=Q\cap D$. 
Thus $Q\cap D\in \QSpec^{\stf}(D)$. Conversely, let $P\in \QSpec^{\stf}(D)$. Then $P\subseteq M$ for some $M\in \QMax^{\stf}(D)$ and $M\Na(D, \star)\in \Max(\Na(D, \star))$. Therefore, $\Na(D, \star)_{M{\tiny \Na}(D, \star)}=D_M(X)$ is a valuation overring of ${\Na}(D, \star)$. Since $D_P(X)$ is an overring of $D_M(X)$ and hence of $\Na(D, \star)$, $D_P(X)=\Na(D, \star)_Q$ for some prime ideal $Q$ of $\Na(D, \star)$. In fact, $Q=PD_P(X)\cap \Na(D, \star)$.   On the other hand, since $Q=(Q\cap D)\Na(D, \star)$, we have $Q=P\Na(D, \star)$, and hence $P\Na(D, \star)\in \Spec(\Na(D, \star))$.  
\end{proof}

Note that the statements of the previous lemma
gene\-ra\-li\-ze
some of the classical
characterizations of Pr\"ufer $v$--multiplication domains
(for
short,  P$v$MD's); for the appropriate references see \cite{FJS-03}.

\bigskip

From Lemmas  \ref{Kr} and \ref{Na=Kr}  we obtain the following.

\begin{corollary}\label{val-star-prime}
Let $D$ a P$\star$MD with quotient field $K$. We denote by $\mathscr{V}^\star(D)$ the set of all the $\star$-valuation overrings of $D$.
Then,
\begin{itemize}
\item[(1)] the canonical map $\Zar(\Na(D, \star)) \rightarrow \Zar(D)$, $W \mapsto W\cap K$, is a continuous surjective map having as its image $\mathscr{V}^\star(D)$.
\item[(2)] the canonical continuous surjective map $\Zar(D) \rightarrow \Spec(D)$ restricted to $\mathscr{V}^\star(D)$ is a bijection with $\QSpec^{\stf}(D)$.
\end{itemize}
\end{corollary}
\begin{proof}  (1) is a direct consequence of Lemma \ref{Kr}(2).

(2) Since $D$ is a  P$\star$MD, $\Na(D, \star)=\Kr(D, \star)$ is a Pr\"ufer domain and so the canonical map $\Zar(\Na(D, \star)) \rightarrow \Spec(\Na(D, \star))$ is a homeomorphism. The conclusion follows from (1) and  Lemma \ref{Na=Kr}(2).
\end{proof}


We recall next some results connecting the Pr\"ufer {\sl semistar} multiplication case with the Pr\"ufer {\sl star} multiplication case.

Let $\star$ be a semistar operation on an integral domain $D$ and let $T$ an overring of $D$; we denote by $\star^{|_T}$ the semistar operation on $T$ obtained  by restriction on $\boldsymbol{\overline{F}}(T) \ (\subseteq \boldsymbol{\overline{F}}(D))$ of $\star: \boldsymbol{\overline{F}}(D)\rightarrow \boldsymbol{\overline{F}}(D)$.

Clearly, if $\star$ has finite type, then so does $\star^{|_T}$ (see \cite[Proposition 2.8]{FL1} or \cite[Example 2.1(e)]{FJS-03}).  By \cite[Proposition 3.1 and Corollary 3.2]{FJS-03}, if $D$ is a P$\star$MD, then $T$ is a P${\star^{|_T}}$MD.
 In case of the overring $T$ of $D$ equal to  $R:= D^\star$ , it is easy to see that $\star^{|_R}$ induces a star operation on $R$, when restricted to $\boldsymbol{{F}}(R) $,  and the following holds.

\begin{lemma}\label{lemma1}
Let $D$ be an integral domain and let $\star$ be a semistar operation on $D$. Let  $R:=D^{\tilde{\star}}$ and let $ {\tilde{\star}}^{|_R}:  {\boldsymbol{F}}(R)\rightarrow  {\boldsymbol{F}}(R) $ be the restriction of $\tilde{\star}$ to ${\boldsymbol{F}}(R)$ (so that ${\tilde{\star}}^{|_R}$ is a star operation on $R$).
Then, $D$ is a P$\star$MD if and only if $R$ is a P${\tilde{\star}}^{|_R}$MD.
\end{lemma}
\begin{proof} See \cite[Proposition 3.3 and Theorem 3.1]{FJS-03}.
\end{proof}

\begin{lemma}\label{lemma2}
Let $R$ be an integral domain and let $\ast$ be a star operation on $R$.
Then, $R$ is a P$\ast$MD if and only if $R$ is a P$v$MD and $\ast_{_{\!f}}= t_R$.
\end{lemma}
\begin{proof}
 See \cite[Proposition 3.4]{FJS-03}. \
 \end{proof}

\begin{lemma}\label{lemma3}
Let $D$ be an integral domain, let $\star$ be a semistar operation on $D$, and let $R=D^{\tilde{\star}}$.
Assume that $D$ is a P$\star$MD. Then, $R$ is a P$v_R$MD and:
$$
{\rm QSpec}^{\star_{_{\!f}}}(D)=\{Q\cap D\mid Q\in {\rm Spec}^{t_R}(R)\}\,,
$$
and $R_Q=D_{Q\cap D}$ for each $Q\in {\rm Spec}^{t_R}(R)$.
\end{lemma}
\begin{proof}
As above, let $\boldsymbol{\ast} := {\widetilde{\star}}^{|_R}:  {\boldsymbol{F}}(R)\rightarrow  {\boldsymbol{F}}(R) $   be the restriction of $\tilde{\star}$ to ${\boldsymbol{F}}(R)$.
Clearly,   $\boldsymbol{\ast}$ is a  star operation on $R$ and $R$ is a   P$\boldsymbol{\ast}$MD (Lemma \ref{lemma1}). Therefore, $R$ is a P$v_R$MD and  $\boldsymbol{\ast_{_{\!f}}}=t_R$ (Lemma \ref{lemma2}).
 Also, since $D$ is a P$\star$MD, $\tilde{\star}=\star_f$ (Lemma \ref{Na=Kr}) and its restriction $\boldsymbol{\ast}$ is of finite type. Thus we have $\boldsymbol{\ast}=\boldsymbol{\ast_{_{\!f}}}=t_R$.
Now let $P$ be a quasi-$\star_{_{\!f}}$-prime ideal of $D$. Then $D_P$ is a valuation domain and $D_P=R_Q$, where $Q:=PD_P\cap R$.
Therefore, $Q$ is a $t_R$-prime ideal of $R$  \cite[Lemma 3.17]{K-89} and $P=Q\cap D$.
Conversely, let $Q$ be a $t_R$-prime ideal of $R$ and let $P:=Q\cap D$. Then $P\subseteq P^{\star_{_{\!f}}}\cap D=(Q\cap D)^{\star_{_{\!f}}}\cap D\subseteq Q^{\star_{_{\!f}}}\cap D=
Q^{\boldsymbol{\ast}}\cap D=
Q^{\boldsymbol{\ast}_{_{\!f}}}\cap D=Q^{t_R}\cap D=Q\cap D=P$. Thus $P$ is a quasi-$\star_{_{\!f}}$-prime ideal of $D$.
\end{proof}

\begin{corollary} \label{Na(D*)=Na(R,t)} Let $D$ be an integral domain, let $\star$ be a semistar operation on $D$,  let $R=D^{\tilde{\star}}$, and let $\boldsymbol{\ast}$ the star operation on $R$
defined by restricting $\widetilde{\star}$ to $\boldsymbol{F}(R)$.
Assume that $D$ is a P$\star$MD. Then the Pr\"ufer domain $\Na(D, \star)$ coincides with $\Na(R, \boldsymbol{\ast} )=\Na(R, v_R)$.
\end{corollary}
\begin{proof}
This is an easy consequence of Lemma \ref{lemma3}, since in the present situation,  $\boldsymbol{\ast}=t_R$ and $\Na(D, \star) =\bigcap\{D_P(X) \mid P \in \QSpec^{\stf}(D)\} =
\bigcap\{R_Q(X) \mid Q \in \Spec^{t_R}(R)\}= \Na(R, t_R)=\Na(R, v_R)$.
\end{proof}


\section{A semistar characterization of sharpness  in Pr\"ufer  domains}\label{prufer}

\smallskip

Given   an integral domain $D$ and a prime ideal  $P \in \Spec(D)$, set
$$
\begin{array}{rl}
\nabla(P) :=& \hskip -8pt  \{ M \in \Max(D) \mid M \nsupseteq P\}\,,\\
\Delta (P) :=& \hskip -7pt \nabla(P)  \cup \{P\}\,,\\
\Theta(P) :=& \hskip -7pt  \bigcap \{ D_M \mid M \in \nabla(P) \}\,.
\end{array}
$$

We use the simpler notation $\nabla$ (respectively, $\Delta$; $\Theta$)  when no possible confusion can arise from the omission of  the prime ideal $P$.
We say that $P$ is {\it sharp}  (or, has the {\it \#-property}) if $\Theta(P) \nsubseteq D_P$  (see \cite[Lemma 1]{G-66} and \cite[Section 1 and Proposition 2.2]{FHL-10}).
\smallskip

 The goal of this section is to provide a characterization  of sharpness using (spectral) semistar operations, at least in some   important classes of integral domains.

Clearly,  for each $P \in \Spec(D)$, we have the following relations among the semistar operations associated to $\nabla$ and to $\Delta$:
$$
{\texttt{s}_{\Delta}}:= {\texttt{s}_{\Delta(P)}} \leq {\texttt{s}_{\nabla(P)}}=: {\texttt{s}_{\nabla}} \,,
$$
$$
 \widetilde{ {\texttt{s}_{\Delta}}}  \lneq
 \widetilde{ {\texttt{s}_{\nabla}}} \; \Rightarrow\; ({\texttt{s}_{\Delta}})_{\!_f}   \lneq    ({\texttt{s}_{\nabla}})_{\!_f} \; \Rightarrow\; {\texttt{s}_{\Delta}}  \lneq     {\texttt{s}_{\nabla}}\,.
 $$


\begin{lemma}
Let $D$ be an integral domain and let $P$ be a nonzero prime ideal of $D$. If $P$ is sharp, then $(s_{\Delta(P)})_{\!_f}\lneq (s_{\nabla(P)})_{\!_f}$. Moreover, assume that $D$ is a finite-conductor domain, i.e., the intersection of any two principal ideals of $D$ is finitely generated. Then $P$ being sharp implies $\widetilde{s_{\Delta(P)}}\lneq \widetilde{s_{\nabla(P)}}$.
\end{lemma}
\begin{proof}  The first statement follows from the fact that $D\in \f(D)$ and $\Theta(P)\not\subseteq D_P$ implies that $D^{s_\Delta}\neq D^{s_\nabla}$.

For the second statement, note first that $\Theta(P)\not\subseteq D_P$ if and only if there exists an element $x$ in the quotient field of $D$ such that $(D:_D xD)\subseteq P$ but $(D:_D xD)\not\subseteq M$ for all $M\in \nabla$.
On the other hand, 
 by Lemma \ref{spectral}, we have $\widetilde{s_\Delta}\neq \widetilde{s_\nabla}$ if and only if $\Clinv(\Delta)\neq \Clinv(\nabla)$, i.e., $P\not\in \Clinv(\nabla)$.  Note here that $P\not\in \Clinv(\nabla)$ is equivalent to   the existence of a finitely generated ideal $J$ of $D$ such that $J\subseteq P$ but $J\not\subseteq M$ for all $M\in \nabla$. Therefore, if $D$ is a finite-conductor domain, then $(D:_D xD)$ is a finitely generated ideal and hence the conclusion follows.
\end{proof}

\begin{proposition} \label{sharp-tilde} Let  $P \in \Spec(D)$. Assume that $D$ is a Pr\"ufer domain. Then, $ P$ is  sharp if and only if
$ \ \widetilde{ {\texttt{s}_{\Delta(P)}}}  \lneq
 \widetilde{ {\texttt{s}_{\nabla(P)}}}  \,.$
\end{proposition}
\begin{proof}
 It is known that, in a Pr\"ufer domain $D$,  $\bigcap \{ D_M \mid M \in \nabla  \} \nsubseteq D_P$  if and only if there exists a finitely generated ideal $J$ of $D$ contained in $P$ but not contained in any $M \in \nabla$ \cite[Corollary 2]{GH-67}. Therefore, the conclusion follows from the above lemma and its proof.
\end{proof}


 In general, the condition ${\texttt{s}_{\Delta(P)}} \lneq {\texttt{s}_{\nabla(P)}}$  does not imply  that $P$ is a  sharp prime ideal (even in the Pr\"ufer domain case),
 as the following example shows.

\begin{example} \label{e:almostded}
{\rm Let $D$ be an almost Dedekind domain with a unique maximal ideal $M_0$ non-finitely generated.
 Then, for $P= M_0$,
$$
\begin{array}{lll}
\nabla:=& \hskip -30 pt {\nabla (M_0)}  = \Max(D) \setminus \{M_0\}\,,  \;\;\;\;\; & \hskip -10 pt  \Delta:= {\Delta(M_0)} = \Max(D)\,, \\
\nabla^{\mbox{\tiny \texttt{gen}}}= &  \hskip -25 pt    \Spec(D) \setminus \{M_0\}\,,   \;\;\;\;\; \;\;\;\;\; \;\;\;\;\;\ & \hskip -10 pt \Delta^{\mbox{\tiny \texttt{gen}}}= \Spec(D)\,,\\
\Clinv(\nabla)= &  \hskip -8 pt  \Clinv(\Delta)  = \Spec(D)\,.
\end{array}
$$
Therefore, 
by Lemma \ref{spectral}, ${\texttt{s}_{\Delta}} \lneq {\texttt{s}_{\nabla}}$, but $ \widetilde{ {\texttt{s}_{\Delta}}}  =
 \widetilde{ {\texttt{s}_{\nabla}}}$ and hence $M_0$ is not a sharp ideal.

 Note also that, in the present situation, $ \Max(D)$ is quasi-compact in $\Spec(D)$ (endowed with the Zariski topology), but  $\Max(D) \setminus \{M_0\}$ is not. Therefore,   by Lemma \ref{spectral},
  $ {\texttt{s}_{\Delta}}$ is of finite type but ${\texttt{s}_{\nabla}}$ is not. In fact, $\widetilde{ {\texttt{s}_{\Delta}}} =({\texttt{s}_{\Delta}})_{\!_f}={\texttt{s}_{\Delta}}=d=\widetilde{ {\texttt{s}_{\nabla}}}=({ {\texttt{s}_{\nabla}}})_{\!_f}\lneq {\texttt{s}_{\nabla}}$.
 }
 \end{example}

 \smallskip

 It is natural to ask if  the condition $({\texttt{s}_{\Delta(P)}})_{\!_f}\lneq    ({\texttt{s}_{\nabla(P)}})_{\!_f}$ implies that $P $ is sharp.
 The answer to this question is  affirmative in the case of Pr\"ufer domains:

\begin{corollary} \label{sharp in prufer} Let  $P \in \Spec(D)$. Assume that $D$ is a Pr\"ufer domain. Then,  the following statements are equivalent:
\begin{itemize}
\item[(i)] $ P$ is   a sharp  prime of $D$.
\item[(ii)]  $ \widetilde{ {\texttt{s}_{\Delta(P)}}}  \lneq
 \widetilde{ {\texttt{s}_{\nabla(P)}}}$.

   \item[(iii)]  $ ({\texttt{s}_{\Delta(P)}})_{\!_f}  \lneq
 ({\texttt{s}_{\nabla(P)}}) _{\!_f}$.
     \item[(iv)]  $ ({\texttt{s}_{\Delta(P)}})_{\!_a}   \lneq
 ({\texttt{s}_{\nabla(P)}})_{\!_a} $. \hfill \qed

\end{itemize}
\end{corollary}
\begin{proof} If $D$ is a Pr\"{u}fer domain, then for any semistar operation $\star$ on $D$, $D$ is a P$\star$MD and hence by Lemma \ref{Na=Kr}(1), $\widetilde{\star} =\stf=\star_a$. Therefore, the conclusion follows from Proposition~\ref{sharp-tilde}.
\end{proof}


\smallskip

In Proposition~\ref{sharp-tilde} (and hence in Corollary \ref{sharp in prufer}), the hypothesis that $D$ is a Pr\"ufer domain cannot be omitted. That is, in general, the condition $\widetilde{\texttt{s}_{\Delta(P)}}\lneq \widetilde{\texttt{s}_{\nabla(P)}}$  does not imply  that $P$ is a  sharp prime ideal (even in the finite-conductor domain case):

\begin{example}{\rm
Let $D:=K[X, Y]$, where $K$ is a field, and let $P$ be a maximal ideal of $D$.
Then $\Delta:=\Delta(P)=\mbox{Max}(D)$ 
and hence $\widetilde{\texttt{s}_{\Delta}}= ({\texttt{s}_{\Delta}})_{\!_f} =\texttt{s}_{\Delta}=d$. Since $D$ is a Krull domain, $D=\bigcap \{D_Q \mid Q\in X^{1}(D) \}$, where $X^{1}(D)$ is the set of height one prime ideals of $D$.
Also, since $D$ is a Hilbert ring and every maximal ideal of $D$ has height 2, each $Q\in X^{1}(D)$ is  contained in infinitely many maximal ideals of $D$ by \cite[Theorem 147]{Kap} (or by \cite[Theorem 30]{Kap}). Thus, we have $Q \subset N$ for some $N \in \nabla:=\nabla(P)=\Max(D) \setminus \{P\}$.
Therefore, $\bigcap \{ D_N \mid N\in \nabla\}\subseteq \bigcap \{ D_Q \mid Q\in X^{1}(D)\}=D$, i.e., $\bigcap \{ D_N \mid N\in \nabla\}\subseteq D_P$. Thus, $P$ is not sharp.
Also, since $P$ is finitely generated,
$$
P^{({\texttt{s}_{\nabla}})_{\!_f}}=P^{{\texttt{s}_{\nabla}}}=\bigcap\{PD_N \mid N\in \nabla\} =\bigcap \{D_N \mid N\in \nabla\} =D\,.
$$
This implies that $d\neq (\texttt{s}_{\nabla})_{\!_f}$ and $P\not\in {\rm QSpec}^{(\texttt{s}_{\nabla})_{\!_f}}(D)$. It is easy to check that $\nabla={\rm QMax}^{(\texttt{s}_{\nabla})_{\!_f}}(D)$. Then $\widetilde{\texttt{s}_{\nabla}}=s_{{\rm QMax}^{(\texttt{s}_{\nabla})_{\!_f}}(D)}=s_\nabla$. Thus, we have $\widetilde{\texttt{s}_{\nabla}}=(\texttt{s}_{\nabla})_{\!_f}=s_\nabla$, and hence $\widetilde{\texttt{s}_{\Delta}}\lneq \widetilde{\texttt{s}_{\nabla}}$.
 }
\end{example}


\bigskip
\section{Sharpness  in Pr\"ufer $\star$-multiplication domains}

\smallskip
The goal of this section is to investigate the notion of sharpness in the more general setting of Pr\"ufer $\star$-multiplication  domains. 

Given a semistar operation $\star$ on an integral domain $D$ and a prime ideal  $P \in \QSpec^{\stf}(D)$, set
$$
\begin{array}{rl}
\nabla^{\stf}(P) :=& \{ M \in \QMax^{\stf}(D) \mid M \nsupseteq P\}\,,\\
\Delta^{\stf}(P) :=& \nabla^{\stf}(P)  \cup \{P\}\,,\\
\Theta^{\stf}(P) :=&  \bigcap \{ D_M \mid M \in \nabla^{\stf}(P) \}\,.
\end{array}
$$
As in the previous section, we use the simpler notation $\nabla^{\stf}$ (respectively, $\Delta^{\stf}$; $\Theta^{\stf}$)  when no possible confusion can arise from the omission of the prime ideal $P$.

We say that $P$ is {\it $\stf$-sharp} (or, {\it $\stf$-\#})  if $\Theta^{\stf}(P) \nsubseteq D_P$.
For example, if $\star =d$  is the identity,  the sharp property  coincides -by definition- with the $d$-sharp property (see, also, \cite{G-64}, \cite{G-66}, \cite{GH-67}, \cite[page 62]{FHP97}, \cite{FHL-10}, \cite[Chapter 2, Section 3]{FHL-13}).

Clearly, for each semistar operation $\star$ and for each $P \in  \QSpec^{\stf}(D)$, we have:
$$
{\texttt{s}_{\Delta^{\!\stf}}}:= {\texttt{s}_{\Delta^{\!\stf}(P)}} \leq {\texttt{s}_{\nabla^{\!\stf}(P)}}=: {\texttt{s}_{\nabla^{\!\stf}}}, \, \mbox{and}
$$
$$ P \mbox{ is} \stf\mbox{-sharp} \;\; \Rightarrow   \;\;  ({\texttt{s}_{\Delta^{\!\stf}(P)}})_{\!_f}  \lneq ({\texttt{s}_{\nabla^{\!\stf}(P)}})_{\!_f}  \;\;  \Rightarrow   \;\;
{\texttt{s}_{\Delta^{\!\stf}(P)}} \lneq {\texttt{s}_{\nabla^{\!\stf}(P)}}\,.
$$
 As we have already observed in Example~\ref{e:almostded}, (even) in the Pr\"ufer domain case  (i.e., when $\stf=d$), the condition ${\texttt{s}_{\Delta^{\!\stf}(P)}} \lneq {\texttt{s}_{\nabla^{\!\stf}(P)}}$ 
  does not imply  that $P$ is a $\stf$-sharp prime ideal.

 The  next goal is to show that, if $D$ is a Pr\"ufer $\star$-multiplication domain, then
 the condition  $({\texttt{s}_{\Delta^{\!\stf}(P)}})_{\!_f}  \lneq ({\texttt{s}_{\nabla^{\!\stf}(P)}})_{\!_f}$ coincides with the property that $P \mbox{ is} \stf\mbox{-sharp}$.

\smallskip

We start by recalling that, if $D$ is a Pr\"ufer domain and $P\in  \Spec(D)$, Gilmer and Heinzer  \cite[Corollary 2]{GH-67}  proved that $P$ is sharp if and only if there exists a finitely generated ideal  $I$ of $D$ such that $I \subseteq P$ and $I \not\subseteq M$ for each $M \in \nabla(P)$.  This was generalized to P$v$MDs in \cite[Theorem 1.6]{ghl}.

\begin{proposition}\label{sharp&fg ideal} Let $D$ be a P$\star$MD and  $P \in \QSpec^{\!\stf}(D)$. Then
$P$ is $\stf$-sharp if and only if $P$ contains a finitely generated ideal $I$ of $D$ such that $I \nsubseteq M$ for each $M \in \nabla^{\stf}(P)$.
\end{proposition}

\begin{proof}
 Suppose that $P$ contains a finitely generated ideal $I$ such that $I \nsubseteq M$ for $M \in \nabla^{\stf}(P)$.
 If $I^{-1}\subseteq D_P$, then $II^{-1}\subseteq ID_P\cap D\subseteq PD_P\cap D=P$, and hence $D=(II^{-1})^{\stf}\cap D\subseteq P^{\stf}\cap D=P$, a contradiction. 
  Therefore,
 $I^{-1} \nsubseteq D_P$.
 Choose $u \in I^{-1} \setminus D_P$, and let $M \in \nabla^{\stf}(P)$.  Then, since $I \subseteq (D:_D u)$, we must have $(D:_D u) \nsubseteq M$, that is, $u \in D_M$.  Hence, $\bigcap \{ D_M \mid M \in \nabla^{\stf}(P)\} = \Theta^{\stf}(P)\not\subseteq D_P$, i.e.,  $P$ is $\stf$-sharp.

 For the converse, suppose that there is an element $v \in \Theta^{\stf}(P) \setminus D_P$, so that $(1,v)^{-1}=(D:_D v) \subseteq P$. 
 Since $(1, v)$ is $\stf$-invertible, { $(1, v)^{-1}$ is $\stf$-finite, i.e.,  there is a finitely generated ideal $J \subseteq (1,v)^{-1}$ for which $J^{\star}=((1,v)^{-1})^{\stf}$ \cite[Proposition 2.6]{fo-pi}.}
 We have $J \subseteq P$.  If $J \subseteq M$ for some $M \in \nabla^{\stf}(P)$, then $(1,v)^{-1} \subseteq J^{\star} \cap D \subseteq M^{\stf} \cap D=M$; however, this contradicts the fact that $v \in D_M$.
\end{proof}

\begin{corollary}
Let $D$ be a P$\star$MD and  $P \in \QSpec^{\!\stf}(D)$, then  the following statements are equivalent.
\begin{itemize}
\item [(i)] $P$ is   $\stf$-sharp.
\item [(ii)]  $\widetilde{ \texttt{s}_{\Delta^{\!\stf}\!(P)}} \lneq \widetilde{ \texttt{s}_{\nabla^{\!\stf}\!(P)}}$.
\item[(iii)] $(\texttt{s}_{\Delta^{\!\stf}\!(P)})_{{\!}_f } \lneq (\texttt{s}_{\nabla^{\!\stf}\!(P)})_{{\!}_f }$.
\item[(iv)] $(\texttt{s}_{\Delta^{\!\stf}\!(P)})_{{\!}_a } \lneq (\texttt{s}_{\nabla^{\!\stf}\!(P)})_{{\!}_a }$.
\end{itemize}
\end{corollary}
\begin{proof} Recall that $\widetilde{\star}=\texttt{s}_{{\tiny\QMax}^{\stf}(D)}=\texttt{s}_{{\tiny \QSpec}^{\stf}(D)}$.
Since $\nabla^{\stf}\subseteq \Delta^{\stf}\subseteq {\QSpec}^{\stf}(D)$, $\texttt{s}_{{\tiny\QSpec}^{\stf}(D)}\leq \texttt{s}_{\Delta^{\stf}}\leq \texttt{s}_{\nabla^{\stf}}$.
Since $D$ is a P$\star$MD, i.e., a P$\widetilde{\star}$MD, $D$ is also a P$\texttt{s}_{\Delta^{\stf}}$MD and a P$\texttt{s}_{\nabla^{\stf}}$MD.
Therefore, by Lemma \ref{Na=Kr}(1), we have that $\widetilde{\texttt{s}_{\Delta^{\stf}}}=(\texttt{s}_{\Delta^{\stf}})_{{\!}_f}=(\texttt{s}_{\Delta^{\stf}})_{{\!}_a}$ and that
$\widetilde{\texttt{s}_{\nabla^{\stf}}}=(\texttt{s}_{\nabla^{\stf}})_{{\!}_f }=(\texttt{s}_{\nabla^{\stf}})_{{\!}_a}$.
Thus, we have the implications   (i)$\Rightarrow$(ii)$\Leftrightarrow$(iii)$\Leftrightarrow$(iv).

For the implication   (ii)$\Rightarrow$(i), recall the fact that $\widetilde{\ \texttt{s}_{\Delta^{\!\stf}}} \neq \widetilde{\ \texttt{s}_{\nabla^{\!\stf}}}$ if and only if $\Clinv(\Delta^{\!\stf})\neq \Clinv(\nabla^{\!\stf})$, i.e., $P\not\in \Clinv(\nabla^{\!\stf})$, which is equivalent to   the existence of a finitely generated ideal $J$ of $D$ such that $J\subseteq P$ but $J\not\subseteq M$ for all $M\in \nabla^{\!\stf}$. Then the conclusion follows   from Proposition \ref{sharp&fg ideal}.
\end{proof}


\smallskip

 An integral domain $D$ is called a {\it $\stf$-sharp domain} (or, {\it $\stf$-\#-domain}) if
each $M \in \QMax^{\stf}(D)$ is  $\stf$-sharp, and it is called a {\it $\stf$-doublesharp domain} (or {\it $\stf$-\#\#-domain}) if each $P \in \QSpec^{\stf}(D)$  is $\stf$-sharp.
A $d$-(double)sharp domain is simply called a {\it (double)sharp domain}. 

We connect the notion of  semistar sharpness (and semistar doublesharpness) to that of star sharpness (and star doublesharpness) of ``special'' overrings.

\begin{proposition}\label{thm}
Let $D$ be an integral domain, let $\star$ be a semistar operation on $D$, and let $R:=D^{\widetilde{\star}}$.
Assume that $D$ is a P$\star$MD. Then,
\begin{enumerate}
\item $D$ is a $\star_{\!_f}$-sharp domain if and only if $R$ is a $t_R$-sharp domain.
\item $D$ is a $\star_{\!_f}$-doublesharp domain if and only if $R$ is a $t_R$-doublesharp domain.
\end{enumerate}
\end{proposition}
\begin{proof} This is a straightforward consequence of Lemma~\ref{lemma3}.
\end{proof}

\medskip

 In 1967 Gilmer and Heinzer proved that a Pr\"ufer domain $D$ is a doublesharp domain if and only if every overring of $D$ is a sharp domain \cite[Theorem 3]{GH-67}.
In order to extend this to the P$\star$MD setting, we need
 to recall a general version of the notion of  linked overring.

Let $D$ be an integral domain, $T$ an overring of $D$, and let $\star, \star'$ be semistar operations on $D, T$, respectively. Following \cite[Section 3]{eBF}, we say that $T$ is {\it  $(\star, \star')$-linked to} $D$ if $F^\star=D^\star$ implies $(FT)^{\star'}=T^{\star'}$ for each nonzero finitely generated integral ideal $F$ of $D$. Clearly, $T$ is   $(\star, \star')$-linked to $D$ if and only if $T$ is   $(\stf, {\star'}{_{\!_f}})$-linked to $D$.
In particular, when $\star = t_D$ and $\star' = t_T$, we say that $T$ is a {\it $t$-linked overring of $D$}.
 Note that $D^{\star}$ is automatically $(\star, \star^{|_{D^\star}})$-linked to $D$, where $\star^{|_{D^\star}}$ is the  (semi)star operation on $D^{\star}$ formed by restricting $\star$ to   ${\boldsymbol{\overline F}}(D^{\star})$.

\smallskip

There is a bijection between the $t$-linked overrings of a Pr\"ufer $v$-multiplication domain and the overrings of its  $t$-Nagata ring. We can generalize this to the case of Pr\"ufer $\star$-multiplication domains.


\begin{theorem} \label{linked&Nagata2}
 Let $D$ be an integral domain with quotient field $K$ and let $\star$ be a semistar operation on $D$. If $D$ is a P$\star$MD, then
the map $T \mapsto \Na(T,t_T)$ from the set of $(\star,t_T)$-linked overrings $T$ of $D$ to the set of overrings of $\Na(D,\star)$ is a bijection with inverse map $\mathcal T \mapsto \mathcal T \cap K$.  In particular, if $R$ is a P$v$MD  with quotient field $K$,  then the map $T \mapsto \Na(T,t_T)$ from the set of $t$-linked overrings of $R$ to the set of overrings of $\Na(R,t_R)$ is a bijection with inverse map $\mathcal T \mapsto \mathcal T \cap K$. \end{theorem}


\begin{proof} We begin by proving the ``in particular'' statement.  Let $T$ be a $t$-linked overring of $R$. Then by \cite[Theorem 3.8 and Corollary 3.9]{K-89}, $T$ is a P$v_T$MD, and it is clear that $\Na(T,t_T)$ is an overring of $\Na(R,t_R)$.  Moreover, $\Na(T,t_T) \cap K = T$ by Lemma~\ref{lemma-Na}(5).  Now let $\mathcal T$ be an overring of $\Na(R,t_R)$.  Then $\mathcal T$ is a Pr\"ufer domain, and hence there is a subset $\Lambda$ of $\Spec^{t_R}(R)$ for which $\mathcal T =\bigcap_{P \in \Lambda} R_P(X)$ \cite[Theorem 26.1]{G}.  It follows that $T:=\mathcal T \cap K = \bigcap_{P \in \Lambda} R_P$.  In particular, $T$ is a P$v_T$MD \cite[Corollary 3.9]{K-89} and is $t$-linked over $R$ \cite[Theorem 3.8]{K-89}.  Moreover, since $\Na(T,t_T) \subseteq R_P(X)$ for $P \in \Lambda$, we have $\Na(T,t_T) \subseteq \mathcal T$.  It remains to show that this inclusion is an equality.  To this end, let $\psi \in \mathcal T$, and write $\psi=g/f$ with $g,f \in T[X]$.  We have $g\Na(T,t_T)=\co(g)\Na(T,t_T)$ and $f\Na(T,t_T)=\co(f)\Na(T,t_T)$ \cite[Lemma 2.11]{K-89}.  Moreover, for any nonzero finitely generated ideal $I$ of $T$, we have $I^{-1}\Na(T,t_T)=(\Na(T,t_T):I\Na(T,t_T))$ \cite[Proposition 2.2]{K-89}.  Since $\co(f)\co(f)^{-1} \Na(T,t_T)=\Na(T,t_T)$, we also have $f^{-1}\Na(T,t_T)=\co(f)^{-1}\Na(T,t_T)$.
Therefore,
$$
\co(g)\co(f)^{-1} \subseteq gf^{-1}\Na(T,t_T) \subseteq gf^{-1}\mathcal T \subseteq \mathcal T,
$$
whence $\co(g)\co(f)^{-1} \subseteq \mathcal T \cap K=T$.  It follows that $\psi=g/f \in \Na(T,t_T)$, as desired.

 For the general statement, let $\star$ be a semistar operation on $D$, and assume that $D$ is a P$\star$MD.  Let $T$ be $(\star,t_T)$-linked to $D$.  Then $\Na(T,t_T) \supseteq \Na(D,\star)$ by \cite[Theorem 3.8]{eBF}. Now let $\mathcal T$ be an overring of $\Na(D,\star)=\Na(R,t_R)$, where $R:=D^{\star}$.  By what has already been proved, $T:=\mathcal T \cap K$ is $t$-linked to $R$, and since $R$ is $(\star,t_R)$-linked to $D$, $T$ is $(\star,t_T)$-linked to $D$ by transitivity \cite[Lemma 3.1(b)]{eBF}.  We also have $\Na(T,t_T)=\mathcal T$.  This completes the proof.
\end{proof}

\smallskip

 We  can now characterize $\stf$-sharpness (or $\stf$-doublesharp\-ness) on $D$ with sharpness  (or doublesharpness) on $\Na(D,\star)$, in the case of Pr\"ufer $\star$-multiplication domains.   Statement (2) generalizes \cite[Theorem 3.6]{ghl}.

\begin{proposition} \label{sharp}
 Let $D$ be a P$\star$MD.
\begin{itemize}
\item[(1)] Let $P \in \QSpec^{\!\stf}(D)$. Then, $P$ is $\stf$-sharp in $D$ if and only if $P\Na(D,\star)$ is sharp in the Pr\"ufer domain $\Na(D,\star)$.
\item[(2)]  $D$ is $\stf$-sharp if and only if $\Na(D,\star)$ is sharp.
\item[(3)]  $D$ is $\stf$-doublesharp if and only if $\Na(D,\star)$ is doublesharp.
\end{itemize}
\end{proposition}
\begin{proof}
(1) Set $T=\Theta^{\stf}(P)=\bigcap \{D_M \mid M\in \nabla^{\stf}(P)\}$.
Then $T$ is a $(\star, t_T)$-linked overring of $D$ and $D_P$ is a $(\star, t_{D_P})$-linked overring of $D$ \cite[Lemma 3.1]{eBF}.
By Theorem \ref{linked&Nagata2}, $T\not\subseteq D_P$ if and only if $\Na(T, t_T)\not\subseteq \Na(D_P, t_{D_P})$.
Note that $\Na(D_P, t_{D_P})=D_P(X)$. We can also show that $\Na(T, t_T)=\bigcap \{D_M(X) \mid M\in \nabla^{\stf}(P) \}$
by the same argument as in the proof of Theorem \ref{linked&Nagata2}.
Thus, we have the equivalence that $\bigcap \{D_M \medskip \mid M\in \nabla^{\stf}(P)\} \not\subseteq D_P$ if and only if $\bigcap \{ D_M(X) \mid M\in \nabla^{\star_f}(P)\} \not\subseteq D_P(X)$, that is, $P$ is $\stf$-sharp in $D$ if and only if $P\Na(D, \star)$ is sharp in $\Na(D, \star)$.

(2) and (3) are direct consequences of (1) and Lemma \ref{Na=Kr}(2).
\end{proof}

One might hope for a correspondence similar to that in Theorem~\ref{linked&Nagata2} for overrings $T$ of a P$\star$MD $D$ that are $(\star,\star')$-linked to $D$ (for some semistar operation $\star'$ on $T$).  However, there are too many such $T$, as the following result shows.
\begin{proposition} \label{p:linked} Let $D$ be a P$\star$MD, where $\star$ is a semistar operation on $D$, let $S$ be a $t$-linked overring of $D^{\star}$, and let $T$ be a ring with $D \subseteq T \subseteq S$. Then, there is a semistar operation $\star'$ on $T$ such that $T$ is $(\star,\star')$-linked to $D$.
\end{proposition}
\begin{proof} Define $\star'$ on $T$ by $E^{\star'}:=(ES)^{t_S}$, for each $E \in \Fb(T)$ \cite[Proposition 2.9]{FL1}.
Let $I$ be a finitely generated ideal of $D$ with $I^{\star}=D^{\star}$.  Then, since $S$ is ($t$-linked over $D^{\star}$ and therefore) $(\star,t_S)$-linked to $D$, $(IT)^{\star'}=(IS)^{t_S}=S=T^{\star'}$; that is, $T$ is $(\star,\star')$-linked to $D$.
\end{proof}

Finally, Gimer-Heinzer's characterization  of doublesharp Pr\"ufer domains b(and its generalization to P$v$MDs \cite[Proposition 2.8]{ghl}) can be extended to the case of
P$\star$MDs as follows.

\begin{corollary} \label{c:##} Let $D$ be a P$\star$MD, where $\star$ is a semistar operation on $D$.  The following are equivalent.
 \begin{enumerate}
\item[(i)] $D$ is a $\stf$-doublesharp domain.
\item[(ii)]  If $T$ is an overring of $D$ and $\star'$ is a semistar operation on $T$ such that $T$ is $(\star,\star')$-linked to $D$, then $T$ is a $\star'{\!{_{\!{_f}}}}$-sharp domain.
\item[(iii)]  If $T$ is a $(\star, t_T)$-linked overring of $D$, then $T$ is a $t_T$-sharp domain.
\end{enumerate}
\end{corollary}
\begin{proof}  Assume that $D$ is a $\stf$-doublesharp  domain.  Then $\Na(D,\star)$ is a doublesharp Pr\"ufer domain by Proposition~\ref{sharp}.  Let $T$ be a $(\star,\star')$-linked overring of $D$.  By \cite[Theorem 3.8]{eBF}, $\Na(T,\star') \supseteq \Na(D,\star)$, whence $\Na(T,\star')$ is a sharp domain.  Since $T$ is a P$\star'$MD \cite[Corollary 5.4]{eBF}, $T$ is a $\star'_{{\!}_f }$-sharp domain by Proposition~\ref{sharp}.  Thus (i) $\Rightarrow$ (ii).

The implication (ii) $\Rightarrow$ (iii) is trivial.

Assume (iii), and let $\mathcal T$ be an overring of the Pr\"ufer domain $\Na(D,\star)$.  By Theorem~\ref{linked&Nagata2}, $T:=\mathcal T \cap K$ is $(\star,t_T)$-linked to $D$ and hence $t_T$-sharp by assumption.  Then, since $\mathcal T =\Na(T, t_T)$ by Theorem~\ref{linked&Nagata2}, $\mathcal T$ is sharp by Proposition~\ref{sharp}.  It follows that $\Na(D,\star)$ is doublesharp, and hence, by Proposition ~\ref{sharp}, that $D$ is $\stf$-doublesharp.  Therefore, (iii) $\Rightarrow$ (i).
\end{proof}


 \end{document}